\documentclass[12pt]{amsart}
\usepackage{amscd,amsmath,amsthm,amssymb}
\usepackage[left]{lineno}
\usepackage{pstcol,pst-plot,pst-3d}
\usepackage{color}
\usepackage{pstricks}
\usepackage{stmaryrd}
\usepackage[utf8]{inputenc}
\usepackage{pstricks-add}
\usepackage{epstopdf}
\usepackage{graphicx}
\newpsstyle{fatline}{linewidth=1.5pt}
\newpsstyle{fyp}{fillstyle=solid,fillcolor=verylight}
\definecolor{verylight}{gray}{0.97}
\definecolor{light}{gray}{0.9}
\definecolor{medium}{gray}{0.85}
\definecolor{dark}{gray}{0.6}

 \def\NZQ{\mathbb}               

 \def\ZZ{{\NZQ Z}}

 \def\FF{{\NZQ F}}
 \def\GG{{\NZQ G}}
 \def\HH{{\NZQ H}}

 \def\frk{\mathfrak}               

 \def\mm{{\frk m}}

 \def\G{{\mathcal G}}

   \def\Fc{{\mathcal F}}

 \def\ab{{\mathbf a}}
 \def\bb{{\mathbf b}}
 \def\xb{{\mathbf x}}

 \def\cb{{\mathbf c}}
 
 \def\eb{{\mathbf e}}
 \def\0b{{\mathbf 0}}

 \def\opn#1#2{\def#1{\operatorname{#2}}} 
 \opn\chara{char} \opn\length{\ell} \opn\pd{pd} \opn\rk{rk}
 \opn\projdim{proj\,dim} \opn\injdim{inj\,dim} \opn\rank{rank}
 \opn\depth{depth} \opn\grade{grade} \opn\height{height}
 \opn\embdim{emb\,dim} \opn\codim{codim}
 
 \opn\Tr{Tr} \opn\bigrank{big\,rank}
 \opn\superheight{superheight}\opn\lcm{lcm}
 \opn\trdeg{tr\,deg}
 \opn\reg{reg} \opn\lreg{lreg} \opn\ini{in} \opn\lpd{lpd}
 \opn\size{size} \opn\sdepth{sdepth}
 \opn\link{link}\opn\fdepth{fdepth}\opn\lex{lex}
 \opn\tr{tr}
 \opn\type{type}
 \opn\gap{gap}
 \opn\arithdeg{arith-deg}
 \opn\HS{HS}
 \opn\div{div} \opn\Div{Div} \opn\cl{cl} \opn\Cl{Cl}
 \opn\Spec{Spec} \opn\Supp{Supp} \opn\supp{supp} \opn\Sing{Sing}
 \opn\Ass{Ass} \opn\Min{Min}\opn\Mon{Mon}
 \opn\Ann{Ann} \opn\Rad{Rad} \opn\Soc{Soc}
 \opn\Im{Im} \opn\Ker{Ker} \opn\Coker{Coker} \opn\Am{Am}
 \opn\Hom{Hom} \opn\Tor{Tor} \opn\Ext{Ext} \opn\End{End}
 \opn\Aut{Aut} \opn\id{id}
 
 \opn\nat{nat}
 \opn\pff{pf}
 \opn\Pf{Pf} \opn\GL{GL} \opn\SL{SL} \opn\mod{mod} \opn\ord{ord}
 \opn\Gin{Gin} \opn\Hilb{Hilb}\opn\sort{sort}
 \opn\PF{PF}\opn\Ap{Ap}
 \opn\mult{mult}
 \opn\bight{bight}
 \opn\aff{aff}
 \opn\relint{relint} \opn\st{st}
 \opn\lk{lk} \opn\cn{cn} \opn\core{core} \opn\vol{vol}  \opn\inp{inp} \opn\nilpot{nilpot}
 \opn\link{link} \opn\star{star}\opn\lex{lex}\opn\set{set}
 \opn\width{wd}
 \opn\Fr{F}
 \opn\QF{QF}
 \opn\G{G}
 \opn\type{type}\opn\res{res}
 \opn\conv{conv}
 \opn\Ind{Ind}
 \opn\gr{gr}
 
 \def\pot#1#2{#1[\kern-0.28ex[#2]\kern-0.28ex]}

 \opn\dirlim{\underrightarrow{\lim}}
 \opn\inivlim{\underleftarrow{\lim}}
 \let\dirsum=\oplus
 
 \let\iso=\cong

 \let\Dirsum=\bigoplus
 
 \let\to=\rightarrow
 
 \def\Implies{\ifmmode\Longrightarrow \else
         \unskip${}\Longrightarrow{}$\ignorespaces\fi}
 \def\implies{\ifmmode\Rightarrow \else
         \unskip${}\Rightarrow{}$\ignorespaces\fi}
 \def\iff{\ifmmode\Longleftrightarrow \else
         \unskip${}\Longleftrightarrow{}$\ignorespaces\fi}

 \let\:=\colon
 \newtheorem{Theorem}{Theorem}[section]
 \newtheorem{Lemma}[Theorem]{Lemma}
 \newtheorem{Corollary}[Theorem]{Corollary}
 \newtheorem{Proposition}[Theorem]{Proposition}
 \newtheorem{Remark}[Theorem]{Remark}
 
 \newtheorem{Example}[Theorem]{Example}
 
 \newtheorem{Definition}[Theorem]{Definition}

 \newtheorem{Questions}[Theorem]{Questions}
 
 \let\epsilon\varepsilon
 \let\kappa=\varkappa
 \def\qed{\ifhmode\textqed\fi
       \ifmmode\ifinner\quad\qedsymbol\else\dispqed\fi\fi}
 \def\textqed{\unskip\nobreak\penalty50
        \hskip2em\hbox{}\nobreak\hfil\qedsymbol
        \parfillskip=0pt \finalhyphendemerits=0}
 \def\dispqed{\rlap{\qquad\qedsymbol}}
 
 \opn\dis{dis}
 \def\pnt{{\raise0.5mm\hbox{\large\bf.}}}
 
 \opn\Lex{Lex}

\begin{document}

\title {Homological shift ideals}

\author {J\"urgen Herzog,  Somayeh Moradi,  Masoomeh Rahimbeigi  and Guangjun Zhu}

\address{J\"urgen Herzog, Fachbereich Mathematik, Universit\"at Duisburg-Essen, Campus Essen, 45117
Essen, Germany} \email{juergen.herzog@uni-essen.de}

\address{Somayeh Moradi, Department of Mathematics, School of Science, Ilam University,
P.O.Box 69315-516, Ilam, Iran}
\email{so.moradi@ilam.ac.ir}

\address{Masoomeh Rahimbeigi, Department of Mathematics, University of Kurdistan, Post
Code 66177-15175, Sanandaj, Iran}
\email{rahimbeigi$_{-}$masoome@yahoo.com}

\address{Guangjun Zhu, School of Mathematical Sciences, Soochow University, Suzhou
215006, P. R. China}
\email{zhuguangjun@suda.edu.cn}

\begin{abstract}
For a monomial ideal $I$, we consider the $i$th homological shift ideal of $I$,  denoted by $\HS_i(I)$,  that is,  the ideal generated by the $i$th multigraded shifts of $I$. Some algebraic properties of this ideal are studied. It is shown that for any monomial ideal $I$ and any monomial prime ideal $P$,  $\HS_i(I(P))\subseteq \HS_i(I)(P)$ for all $i$, where $I(P)$ is the monomial localization of $I$.
In particular, we consider the homological shift ideal of some families of monomial ideals with linear quotients. For any $\cb$-bounded principal Borel ideal $I$ and for the edge ideal of complement of any path graph, it is proved that $\HS_i(I)$ has linear quotients for all $i$.
As an example of $\cb$-bounded principal Borel ideals, Veronese type ideals are considered and it is shown that the homological shift ideal of  these ideals are polymatroidal. This implies that for any polymatroidal ideal which satisfies  the strong exchange property, $\HS_j(I)$ is again a  polymatroidal ideal for all $j$.
 Moreover, for any edge ideal with  linear resolution, the ideal $\HS_j(I)$ is characterized and it is shown that $\HS_1(I)$ has linear quotients.
\end{abstract}

\keywords{$\cb$-bounded Borel ideals, monomial ideals, multigraded shifts, polymatroidal ideals.}

\subjclass[2010]{Primary 13F20; Secondary 13H10}

\maketitle

\setcounter{tocdepth}{1}

\section*{Introduction}

Let $K$ be a field,  $I$  a monomial ideal in  the polynomial ring $S=K[x_1,\ldots,x_n]$ and let $\FF$ be its minimal  multigraded free $S$-resolution. Then $F_j=\Dirsum_{k=1}^{b_j} S(-{\ab_{jk}})$ with   each $\ab_{jk}$  an integer vector in $\mathbb{Z}^n$ with non-negative entries. A monomial $x_1^{a_1}\cdots x_n^{a_n}$ will be denoted by $\xb^{\ab}$ where $\ab=(a_1,\ldots,a_n)$. With this notation introduced we define the $j$th {\em homological shift ideal} of $I$ to be the monomial ideal $\HS_j(I)$ generated by the monomials $\xb^{\ab_{jk}}$, $k=1,\ldots,b_j$.

When Somayeh Bandari  and Shamila Bayati met the first author in Essen in 2012,  the question came up whether the homological shift ideals of a polymatroidal ideal are again polymatroidal, as many examples indicated.  This question is still open. However Bayati \cite{Ba} gave a positive answer to this question when the polymatroidal ideal is actually matroidal. In Section~3 of this paper we give an affirmative answer to this question in another special case, namely for polymatroidal ideals which satisfy the strong exchange property.

In this paper we ask  ourselves more generally which properties of $I$ are shared by the ideals $\HS_j(I)$. So far all known cases and examples indicate that $\HS_j(I)$ has linear resolution or linear quotients if $I$ has the corresponding property.

In Section~1 we discuss general properties of homological shift ideals. It should be mentioned that $\HS_j(I)$ does not always give us the full information about the multigraded shifts of $F_j$. Indeed, some multigraded shift may appear with  higher multiplicity, which is not reflected by $\HS_j(I)$, or other  shifts do not appear in the minimal monomial set of generators of $\HS_j(I)$.  The latter  does not happen if $I$ has linear resolution.

In the case that $I$ is generated in a single degree and has linear relations, the generators of $\HS_1(I)$ can be quite explicitly described, as explained in Proposition~\ref{HS1}. If we require that $I$ has linear quotients, then it can be shown that $\HS_{j+1}(I)\subseteq \HS_1(\HS_j(I))$, see Proposition~\ref{easy}. In general however equality does not hold. Indeed, if $p=\projdim(I)$ and $\HS_p(I)$ is not a principal ideal, then $\HS_1(\HS_p(I))\neq 0$, while $\HS_{p+1}(I)=0$. Even worse, if $I$ does not have linear  quotients then   not even the inclusion   $\HS_{j+1}(I)\subseteq \HS_1(\HS_j(I))$ may be valid.  It would be of interest to have precise conditions for when $\HS_{j+1}(I)= \HS_1(\HS_j(I))$ for all $j$.

It is quite easy to see (Proposition~\ref{contained})  that $\HS_{j+1}(I)\subseteq \mm \HS_{j}(I)$, and consequently,  $\height(\HS_{i+1}(I))\leq \height(\HS_{i}(I))$ for all $j$, as noted in Corollary~\ref{containedh}. Due to this fact one would expect  that $\projdim \HS_{j+1}(I)\leq  \projdim \HS_{j}(I)$. So far we know this only in a special case which will be described below.

Given a monomial ideal $I$ and an integer vector in $\cb\in\mathbb{Z}^n$ with non-negative entries, one defines the ideal  $I^{\leq \cb}$ which is obtained from $I$ by taking only those monomials among the generators of $I$ whose exponents are componentwise less than or equal to $\cb$. It is shown in Corollary~\ref{hsjrestriction} that $\HS_j(I^{\leq \cb})=\HS_j(I)^{\leq \cb}$. Another important operation on monomial ideals is monomial localization. Here we show that $\HS_i(I(P))\subseteq \HS_i(I)(P)$ for all $i$ and all monomial prime ideals $P$. For the proof we construct,  starting with the resolution $\FF$ of $I$,  a multigraded (in general non-minimal) multigraded free resolution for $I(P)$, see Lemma~\ref{local}. This construction is of interest in its own. Ordinary localization of $\FF$ would destroy its multigraded structure. In Proposition~\ref{last} we observe that $\HS_{n-1}(I)$ can be expressed in terms of the socle of $I$ and in Proposition~\ref{pol} it is shown that $\HS_j(I^\wp)=\HS_j(I)^\wp$, where for a monomial ideal $J$, the ideal $J^\wp$ denotes its polarization.

In Section~2 we study the homological shift ideals of $\cb$-bounded strongly stable ideals.
There we generalize a result obtained by Bayati et al \cite{BJT} who showed that the homological shift ideals of a principal strongly stable ideal have linear quotients. We extend this result to $\cb$-bounded strongly stable ideals, see Theorem~\ref{mainHSj}. For the proof we show in Lemma~\ref{principalbound} that  $\cb$-bounded principal strongly stable ideals are obtained from principal strongly stable ideals by simply bounding the exponents of its generators by $\cb$. This is not as obvious as it appears at the first glance. Having this and a characterization of linear quotient ideals, as described in Lemma~\ref{criterion}, the desired result can be deduced from \cite[Theorem 2.4]{BJT}.

In Section 3 we consider the homological shift ideals of ideals of Veronese type, denoted $I_{\cb,n,d}$. These are special classes of polymatroidal ideals, namely those consisting of all monomials of degree $d$ in $n$ variables whose exponents are bounded by $\cb$. Any polymatroidal ideal satisfying the strong exchange property is of this type, up to multiplication by a monomial. It is shown in Proposition \ref{veroneseborel} that  any Veronese type ideal is indeed a $\cb$-bounded principal strongly stable ideal and then it has linear quotients.  Moreover, it is proved that the homological  shift ideals of $I_{\cb,n,d}$ are all polymatroidal, see Theorem~\ref{supp}.

Finally, in Section~4  the homological shift ideals of edge ideals with linear resolution are studied. Let $G$ be a finite  simple graph. By Fr\"oberg \cite{Fr}, the edge ideal $I(G)$ of $G$ has linear resolution if and only if the complementary graph $G^c$ of $G$ is chordal. By a theorem of Dirac \cite{Dirac}, a chordal graph has a perfect elimination ordering.  In terms of this elimination ordering, Theorem~\ref{chordal} provides for all $j$   an explicit description of the generators of $\HS_j(I(G^c))$, when $G$ is chordal. In the specific case of a path graph $P_n$, we show in Proposition~\ref{pathlinear} that $\HS_j(I(P_n^c))$ has linear quotients for all  $j\geq 0$. In Corollary~\ref{pathk} it is shown that $\HS_j(I(P_n^c))\neq (0)$ if and only if $j\leq n-3$. This is the consequence of a more general result on proper interval graphs, described in Lemma~\ref{interval}. Proposition~\ref{projdim} provides a formula for the projective dimension of $\HS_j(I(P_n^c))$. Indeed, we show that $\HS_j(I(P_n^c))=n-j-2$ for $j\leq n-3$. This supports our expectation that quite generally one has $\projdim\HS_{j+1}(I)\leq \projdim\HS_{j}(I)$ for all $j$. Even though, we know the generators of  $\HS_j(I(G^c))$ when $G$ is chordal, at present we are not able to prove that these ideals all have linear quotients. However, we show in  Theorem~\ref{linear quotients} that $\HS_1(I(G^c))$ is a vertex splittable ideal and hence has linear quotients.

\section{General  properties of   homological shift ideals}\label{sec1}

Let $K$ be a field and  $S=K[x_1,\ldots,x_n]$  be the polynomial ring in $n$ indeterminates over $K$.
\begin{Definition}
{\em Let
 $I\subset S$ be a monomial ideal with minimal multigraded free $S$-resolution
\[
\FF: 0\longrightarrow F_q\stackrel{\partial{q}}\longrightarrow F_{q-1}\longrightarrow\cdots \longrightarrow F_1\stackrel{\partial{1}}\longrightarrow F_0\longrightarrow I\longrightarrow 0,
\]
where $F_i=\Dirsum_{j=1}^{b_i}S(-\ab_{ij})$. The vectors $\ab_{ij}$ are called the {\em multigraded shifts} of the resolution  $\FF$.
The monomial ideal $\HS_i(I)=(\xb^{\ab_{ij}}\: j=1,\ldots,b_i)$ is called the {\em $i$th homological shift ideal} of $I$.}
\end{Definition}

Note that $\HS_0(I)=I$.

\begin{Example}
\label{exm1}
{\em
(a) Let $I=(x_1^3,x_1^2x_2,x_1x_2^2)$. Then $\HS_1(I)=(x_1^3x_2,x_1^2x_2^2)$ and $\HS_2(I)=0$.

(b) Let $I$ be a complete intersection with $\mathcal{G}(I)=\{u_1,\ldots,u_m\}$, where   $\mathcal{G}(I)$ is the unique minimal set of monomial generators of $I$. Since the Koszul complex with respect to the sequence  $u_1,\ldots,u_m$ is the minimal multigraded free resolution of $I$, we obtain
\[
\HS_j(I)=(u_{i_1}u_{i_2}\cdots u_{i_j}\:\; 1\leq i_1<i_2<\cdots<i_j\leq m).
\]

(c) Let $I$ be a Gorenstein monomial ideal. Let $p$ be the projective dimension of $S/I$ and  $\FF$ its  graded minimal free resolution. Then $F_p=S(-\ab)$ and $\FF$ is self-dual. This follows \cite[Corollary 3.3.9 and Theorem 3.3.7]{BH}. As a consequence, if $F_j=\Dirsum_{k=1}^{b_j}S(-\ab_{jk})$, then
\[
\{\ab_{p-j,1},\ldots, \ab_{p-j,b_{p-j}}\}=\{\ab-\ab_{j,1},\ldots, \ab-\ab_{j,b_j}\}.
\]
Thus for all $j$,
$$\mathcal{G}(\HS_{p-j-1}(I))\subseteq \{v/u\:\, u\in \mathcal{G}(\HS_j(I))\},$$ where $v=\xb^{\ab}$. Equality holds, when $\mathcal{G}(\HS_j(I))=\{\xb^{\ab_{j1}},\ldots,  \xb^{\ab_{jb_j}}\}$.

(d) Let $I$ be a stable monomial ideal. By using the Eliahou-Kervaire resolution \cite{EK} of $I$ , we have that
\[
\HS_j(I)=(x_Fu\:\;  u\in \mathcal{G}(I),\; F\subset [n], \max(F)< m(u), \text{ and $|F|=j$}) .
\]
Here, $x_F=\prod\limits_{i\in F} x_i$ $\max(F) =\max\{i\:\; i\in F\}$ and $m(u)=\max\{i\:\; \text{ $x_i$ divides $u$}\}$.
}
\end{Example}

The following result describes $\HS_1(I)$ when $I$ is linearly related.

\begin{Proposition}
\label{HS1}
Let $I$ be generated in a single degree, and suppose that $I$ has linear relations.  Then $\HS_1(I)$ is generated by all monomials of the form $x_iu$ with $u\in \mathcal{G}(I)$ for which there exist $j\neq i$ and $v\in \mathcal{G}(I)$ such that $x_iu =x_jv$.
\end{Proposition}

\begin{proof}
Let $\mathcal{G}(I)=\{u_1,\ldots,u_m\}$, $F$ be the free $S$-module with basis $e_1,\ldots,e_m$, and let $\varphi\: F\to I$ be the $S$-module homomorphism  with $\varphi(e_i)=u_i$ for $i=1,\ldots,m$. Then the multi-degree of $e_i$ is the same as that of $u_i$. It is clear $x_ie_k -x_je_l\in \Ker(\varphi)$ if $x_iu_k =x_ju_l$. Therefore, $x_iu_k\in \HS_1(I)$.

Conversely, let $r\in \Ker(\varphi)$ be a non-zero minimal generator of multi-degree $\ab$. By our assumption,  $\xb^{\ab}$ is of degree $d+1$ if $I$ is generated in degree $d$. Then $r$ is of the form $\sum_{k=1}^t\lambda_k (\xb^\ab/u_{i_k})e_{i_k}$ where the coefficients $\lambda_k$ belong to $K$, and where the  sum is  taken over all $u_{i_k}$ which divide $\xb^\ab$.

Since $r\neq 0$, the sum of $r$ has at least two non-zero terms, say, $\lambda_1,\lambda_2\neq 0$. Let $s=(\xb^{\ab}/u_{i_1})e_{i_1}- (\xb^{\ab}/u_{i_2})e_{i_2}$. Then $s\in \Ker(\varphi)$ is a  relation as described in the statement of the proposition and $r-\lambda_1s\in  \Ker(\varphi)$. If $r-\lambda_1 s=0$, we are done. Otherwise, $r-\lambda_1 s$ has at most  $t-1$ summands, and we may apply induction to deduce the desired conclusion.
\end{proof}

Let $I$ be a monomial ideal generated in degree $d$.  The ideal $I$ is said to have {\em linear quotients} if there exists some order $u_1,\ldots,  u_m$ of the elements of $\mathcal{G}(I)$ such that each  colon ideal $(u_1,\ldots,u_{j-1}) : u_j$ is  generated by a subset of $\{x_1, \ldots,  x_n\}$ for all $j=2, \ldots, m$.

The situation described in Example~\ref{exm1}(c) holds more generally. Indeed, let $I$ be any monomial ideal with $\mathcal{G}(I)=\{u_1,u_2,\ldots, u_m\}$. Suppose that $\deg u_1\leq \deg u_2\leq \cdots \leq \deg u_m$ and that $I$ has linear quotients with respect to this order of monomials. For  $u_j\in \mathcal{G}(I)$, let $(u_1,\ldots,u_{j-1}):u_j=\{x_{i_1},\ldots,x_{i_k}\}$. We denote by   $\set_I(u_j)$ or simply $\set(u_j)$  the set $\{i_1,\ldots,i_k\}$. It follows from \cite[Lemma 1.5]{HT} that
\begin{eqnarray}\label{eq1}
 \HS_j(I) &=& (x_Fu\:\; u\in \mathcal{G}(I), |F|=j, F\subseteq \set(u)).
\end{eqnarray}

\medskip
This observation has the following consequence.

\begin{Proposition}
\label{easy}
Suppose that the monomial ideal $I$ has linear quotients. Then
\[
\HS_{j+1}(I)\subseteq \HS_1(\HS_j(I))
\]
for all $j$.
\end{Proposition}

\begin{proof}
The assertion is trivial for $j=0$. Now let us assume that $j>0$. We adopt the notation of the above discussion. Let $x_Fu\in \HS_{j+1}(I)$ with $u\in \mathcal{G}(I)$ and $F\subseteq \set(u)$. Since $j\geq 1$, there exist $a,b\in F$ with $a\neq b$. Since $F\setminus\{a\}\subset \set(u)$ and  $F\setminus\{b\}\subset \set(u)$, it follows that
$x_{F\setminus\{a\}}u, x_{F\setminus \{b\}}u\in \HS_j(I)$, and since $x_Fu= x_ax_{F\setminus \{a\}}u=x_bx_{F\setminus \{b\}}u$,  we deduce that $x_Fu\in \HS_1(\HS_j(I))$, as desired.
\end{proof}

In general, the inclusion  $\HS_{j+1}(I)\subseteq \HS_1(\HS_j(I))$ does  not hold. Indeed, let $I=(x_1x_2,x_2x_3,x_3x_4, x_4x_5,x_5x_1)$.  Then $$\HS_1(I)= (x_1x_2x_3, x_2x_3x_4, x_3x_4x_5, x_1x_4x_5,x_1x_2x_5)$$ and $$\HS_1(\HS_1(I))= (x_1x_2x_3x_4, x_2x_3x_4x_5, x_1x_3x_4x_5, x_1x_2x_4x_5),$$ while  $\HS_2(I)=(x_1x_2x_3x_4x_5)$.

\medskip
One may expect that $\HS_{j+1}(I)= \HS_1(\HS_j(I))$ when $I$ has linear quotients. However this is not the case. The simplest example of this kind is the ideal $I=(x_2x_4, x_1x_2,x_1x_3)$ which has linear quotients for the generators in the given order. Here we have $\HS_1(I)=(x_1x_2x_3, x_1x_2x_4)$, and $\HS_1(\HS_1(I))=(x_1x_2x_3x_4)$. However, $\HS_2(I)=(0)$.

\medskip
Let  $u\in S$ be a monomial, then $u= x_1^{a_1}\cdots x_n^{a_n}$,   and we write $u=\xb^{\ab}$ where $\ab=(a_1,\ldots,a_n)$. Let $\mm=(x_1,\ldots,x_n)$ be the graded maximal ideal of $S$.

\begin{Proposition}\label{contained}
 For all $j$ we have $\HS_{j+1}(I)\subseteq \mm \HS_{j}(I).$
\end{Proposition}

\begin{proof} Let
\[
\FF: 0\longrightarrow F_p\stackrel{\partial{p}}\longrightarrow F_{p-1}\longrightarrow\cdots \longrightarrow F_1\stackrel{\partial{1}}\longrightarrow F_0\longrightarrow I\longrightarrow 0
\]
be  the minimal multigraded free $S$-resolution of $I$,  where $F_j=\Dirsum_{k}S(-\ab_{jk})$.
Then the image of the differential $\partial_{j+1}: F_{j+1}\to F_j$ is contained in $\mm F_j$ for any $j$.
Let $\xb^{\ab}\in \HS_{j+1}(I)$,  then there exists a basis element $f\in F_{j+1}$ such that $\deg(f)=\ab$. Assume that the differential $\partial_{j+1}: F_{j+1}\to F_j$
  is given by
  $$
  \partial_{j+1}(f)=\sum\limits_{\ell\in A}\lambda_\ell\xb^{\cb_{\ell}}f_\ell
  $$
  where $\lambda_\ell\in K$ and $\{f_\ell: \ell\in A\}$ is a minimal multigraded basis of $F_j$. Then $\xb^{\cb_{\ell}}\in \mm$, and   it follows that
  $\xb^{\ab}=\xb^{\cb_{\ell}}\xb^{\deg(f_{\ell})}$ for such $\ell$ with $\lambda_\ell\neq 0$. Thus $\xb^{\ab}\in \mm\HS_j(I)$, since $\xb^{\deg(f_{\ell})}\in \HS_j(I)$.
\end{proof}

An  immediate consequence of the above proposition  is

\begin{Corollary}\label{containedh}
 For all $i$ we have $\height(\HS_{i+1}(I))\leq \height(\HS_{i}(I)).$
\end{Corollary}

\begin{Proposition}
\label{localization} Let $S_1=K[x_1,\ldots, x_m]$, $S_2=K[x_{m+1},\ldots, x_n]$ be two polynomial rings in disjoint sets of indeterminates and
$I$ and  $J$  be  monomial ideals of $S_1$ and $S_2$, respectively. Let $S=K[x_1,\ldots, x_n]$.
Then  for any $i$
$$\HS_i(IJ)=\sum\limits_{k+\ell=i}\HS_k(I)\HS_\ell(J).$$
\end{Proposition}
\begin{proof} Let
\[
\FF: 0\to F_p\to\cdots \to F_1\to F_0\to I\to 0
\]
and
\[
\GG: 0\to G_q\to\cdots \to G_1\to G_0\to J\to 0
\]
be  the minimal multigraded free $S_1$-resolutions of $I$ and $S_2$-resolutions of $J$,  respectively.
Let $H_i=\sum\limits_{k+\ell=i}F_k \otimes_K G_\ell$, then
\[
\HH: 0\to H_{p+q}\to\cdots \to H_1\to H_0\to IJ\to 0
\]
is the minimal multigraded free $S$-resolutions of $IJ$. The desired conclusion follows, because if
$F_k=\Dirsum_iS_1(-\ab_i)$ and $G_\ell=\Dirsum_jS_2(-\bb_j)$, then $F_k \otimes_K G_\ell=\Dirsum_{i,j}S(-\ab_i-\bb_j)$.
\end{proof}

 Let $\cb=(c_1,\ldots,c_n)$ be an integer vector with $c_i\geq 0$.  A  monomial $u=\xb^{\ab}$  is called $\cb$-bounded, if $\ab\leq \cb$, that is, $a_i\leq c_i$  for all $i$. Let $I$ be a monomial ideal generated by the monomials $u_1,\ldots, u_m$.  We set
\[
I^{\leq \cb}=(u_i\:\; \text{ $u_i$  is  $\cb$-bounded}).
\]

\begin{Remark}
\label{easy2}
{\em The  definition of $I^{\leq \cb}$  does not depend on the chosen  set of monomial generators $u_1,\ldots,u_m$ of $I$. Indeed, let $J$ be the ideal generated by the $\cb$-bounded monomials of $\mathcal{G}(I)$.
We may assume that $\mathcal{G}(J)=\{u_1,\ldots,u_r\}$ with $r\leq m$. We want to show that $J=I^{\leq \cb}$. Obviously $J\subseteq I^{\leq \cb}$. Now let $w\in I^{\leq\cb}$. Then there exists $1\leq i\leq m$ such that $u_i|w$. Since $w$ is $\cb$-bounded, then $u_i$ is also $\cb$-bounded. Thus $i\leq r$. This means that $u_i\in J$ and hence $w\in J$.}
\end{Remark}

We  recall the so-called restriction lemma.

\begin{Lemma}[\cite{HHZ}, Lemma 4.4]
\label{restriction}
Let $I\subset S$ be a monomial ideal, and  $\FF$ be its minimal multigraded free $S$-resolution. Furthermore, let  $\cb=(c_1,\ldots,c_n)$ be an integer vector with $c_i\geq 0$. Let $F_i=\Dirsum_{j}S(-\ab_{ij})$. Then $\FF^{\leq \cb}$ with $F_i^{\leq\cb}= \Dirsum_{ \ab_{ij}\leq \cb}S(-\ab_{ij})$ is a subcomplex of $\FF$ and a minimal multigraded free resolution of $I^{\leq \cb}$.
\end{Lemma}

The restriction lemma  together with Remark~\ref{easy2} implies

\begin{Corollary}
\label{hsjrestriction}
$\HS_j(I^{\leq \cb})=\HS_j(I)^{\leq \cb}$.
\end{Corollary}

Let $P$ be a monomial prime ideal. The monomial localization $I(P)$  of $I$ is obtained from $I$ by the substitution $x_i\mapsto 1$ for $x_i\not\in P$.   Note that $I(P)$ is a monomial ideal
in $S(P)$ and  $I(P)S_P=IS_P$, where $S(P)$ is the polynomial ring in the variables which generate $P$ and $S_P$ is the localization  of $S$ with respect to $P$.

\begin{Proposition}
\label{localization}
$\HS_i(I(P))\subseteq \HS_i(I)(P)$ for all $i$ and all monomial prime ideals $P$.
\end{Proposition}

For the proof of   Proposition~\ref{localization}  we need monomial localization of multigraded free resolutions. Let $I\subset S$ be a monomial ideal, and
let
\[
\FF: 0\to F_q\to\cdots \to F_1\to F_0\to 0
\]
be a  multigraded free $S$-resolution of $I$. Let  $F_i=\Dirsum_jS(-\ab_{ij})$. Since $S(-\ab_{ij})$ is isomorphic to the principal ideal $(\xb^{\ab_{ij}})$, we may write $F_i=\Dirsum_j(\xb^{\ab_{ij}})$. For example if $I=(x_1^2,x_1x_2)$, then with this notation,
\begin{eqnarray}
\label{simpleexample}
0\to (x_1^2x_2)\to(x_1^2)\dirsum (x_1x_2)\to (x_1^2,x_1x_2)\to 0
\end{eqnarray}
is the minimal multigraded free $S$-resolution of $I$, and $\partial(x_1^2x_2)=x_2(x_1^2)-x_1(x_1x_2).$

We let  $F_i(P)=\Dirsum_j(\xb^{\ab_{ij}})(P)=\Dirsum_j(\xb^{\bb_{ij}})$, where $\xb^{\bb_{ij}}$ is obtained from $\xb^{\ab_{ij}}$   by the substitution $x_i\mapsto 1$ for $x_i\not\in P$.

In order to simplify notation, for any monomial $u\in S$ we let $u^*$ be the monomial which is obtained from $u$  by the substitution $x_i\mapsto 1$ for $x_i\not\in P$. We now define the complex
\[
\FF(P): 0\to F_q(P)\to\cdots \to F_1(P)\to F_0(P)\to 0,
\]
whose differential $\partial^*$ is given as follows: set $u_{ij}=\xb^{\ab_{ij}}$ for all $i$ and $j$.  Then
$\partial(u_{ij})=\sum_k\lambda_{jk}v_{jk}u_{i-1,k}$ with $\lambda_{jk}\in K$ and $v_{jk}$ monomials such that $u_{ij}=v_{jk}u_{i-1,k}$ for all $k$.

Then we set
\[
\partial^*(u_{ij}^*)=\sum_k\lambda_{jk}v_{jk}^*u_{i-1,k}^*.
\]

\begin{Lemma}
\label{local}
With the notation introduced,  $\FF(P)$ is a multigraded free $S$-resolution of $I(P)$.
\end{Lemma}

\begin{proof}
Let $\ab\in \ZZ^n$ be an integer vector with non-negative entries. The $\ab$-graded piece of $\FF$ is a complex
\[
\FF_\ab\: 0\to (F_q)_\ab\to \cdots \to (F_1)_\ab\to (F_0)_\ab\to 0
\]
of finite dimensional $K$-vector spaces. To say that $\FF_\ab$ is a resolution of $I_\ab$, is equivalent to say that $\FF_\ab\to I_\ab\to 0$ is exact for all $\ab$.

We will show that the complexes of $K$-vector spaces
\begin{eqnarray}
\label{first}
0\to (F_q)_\ab\to \cdots \to (F_1)_\ab\to (F_0)_\ab\to I_\ab\to 0,
\end{eqnarray}
and
\begin{eqnarray}
\label{second}
0\to F_q(P)_{\ab^*}\to \cdots \to F_1(P)_{\ab^*}\to F_0(P)_{\ab^*}\to I(P)_{\ab^*}\to 0,
\end{eqnarray}
are isomorphic for all $\ab$. Here $\ab^*$ is defined by the equation $u^*=\xb^{\ab^*}$,  where $u=\xb^\ab$.

By showing this, since the complexes (\ref{first}) are exact for all $\ab$, the complexes (\ref{second}) are also exact for all $\ab$. This then shows that $\FF(P)$ is multigraded free resolution of $I(P)$, as desired.

The isomorphism between complex (\ref{first}) and complex (\ref{second}) is  established as follows: let $u=\xb^{\ab}$. Then $(F_i)_\ab=\Dirsum Kv_{ij}(u_{ij})$ with monomials $v_{ij}$, where the sum is taken over all $j$ with $v_{ij}(u_{ij})=u$. Similarly, $(F_i(P))_{\ab^*}=\Dirsum Kw_{ij}(u_{ij}^*)$ with monomials $w_{ij}$, where the sum is taken over all $j$ with $w_{ij}(u_{ij}^*)=u^*$. Since $u^*=(v_{ij}u_{ij})^*= v_{ij}^*u_{ij}^*$, it follows that $w_{ij}=v_{ij}^*$ for all $j$. Thus for each $i$ the map $\varphi\: (F_i)_\ab\to (F_i(P))_{\ab^*}$ with $v_{ij}(u_{ij})\mapsto v_{ij}^*(u_{ij}^*)$ is an isomorphism of $K$-vector spaces. It is readily seen, that for all $i>0$ the  sequence of maps $\varphi =(\varphi_i)_{i=0,\ldots,q}$ commutes with the differentials of $(\FF)_\ab$ and $\FF(P)_{\ab^*}$. Thus $\FF(P)_{\ab^*}$ is indeed an   exact  complex, and $\varphi$ is a complex homomorphism. Moreover,  $\varphi_0$   is compatible with the augmentation maps $(F_0)_\ab\to I_\ab$ and $(F_0(P))_{\ab^*}\to I(P)_{\ab^*}$. This completes the proof.
\end{proof}

In the example given by (\ref{simpleexample}), if $P=(x_1)$, then
\[
\FF(P): 0\to (x_1^2)\to (x_1^2)\dirsum (x_1)\to (x_1)\to 0,
\]
and $\partial^*(x_1^2)=1(x_1^2)-x_1(x_1)$. This example shows that $\FF(P)$ is in general no longer a minimal free resolution of $I(P)$, even though $\FF$ may be minimal.

\begin{proof}[Proof of Proposition~\ref{localization}]
Note that $\HS_i(I)(P)=(\xb^{\ab_{i1}}, \xb^{\ab_{i2}}, \ldots)(P)=(\xb^{\bb_{i1}}, \xb^{\bb_{i2}}, \ldots)$. Thus   the generators of $\HS_i(I)(P)$ correspond the $i$th homological shifts of $\FF(P)$ which is  a multigraded, but not necessarily a minimal  free resolution of $I(P)$. This implies that
$\HS_i(I(P))\subseteq \HS_i(I)(P)$.
\end{proof}

In our example given in (\ref{simpleexample}), we have $\HS_1(I)(P)=(x_1^2)$ and $\HS_1(I(P))=0$.

\medskip
For the last homological shift ideal of a monomial ideal we have the following interpretation.

\begin{Proposition}
\label{last}
Let $I\subset S$ be a monomial ideal, and $\mm$ be the graded maximal ideal of $S$. Furthermore, let  $(I:\mm)/I=\Dirsum_{i=1}^tK(-\ab_i) $ and $J=(\xb^{\ab_i}\:\; i=1,\ldots,t)$.  Then $\HS_{n-1}(I)=x_1x_2\cdots x_n J$.
\end{Proposition}

\begin{proof}
Let
\[
0\to F_{n-1}\to \cdots \to F_1\to F_0\to I\to 0
\]
be the minimal   multigraded free $S$-resolution of $I$ with  $F_{n-1}= \Dirsum_{i=1}^{r} S(-\ab_i)$. Then
 $\HS_{n-1}(I) =(\xb^{\ab_1},\ldots,\xb^{\ab_r})$.

There exists the following isomorphisms of multigraded graded modules.
\[
\Dirsum_{i=1}^rK(-\ab_i)\iso \Tor_{n-1}(K, I)\iso \Tor_n(K, S/I)\iso H_n(x_1,\ldots,x_n;S/I).
\]
Here $H_n(x_1,\ldots,x_n;S/I)$ denotes the $n$th Koszul homology of $S/I$ with respect to the sequence $x_1,\ldots,x_n$. Since the $K$-vector space $H_n(x_1,\ldots,x_n;S/I)$ admits a basis consisting of  the  elements $(u+I)e_1\wedge e_2\wedge \cdots \wedge e_n$ with $u\in \mathcal{G}(J)$, a comparison of multidegrees provides the desired result.
\end{proof}

Via polarization many questions about monomial ideals can be reduced to questions about squarefree monomial ideals. Let $u=x_1^{a_1}\cdots x_n^{a_n}$ be a monomial.  We define the {\em polarized monomial} of $u$ to be the squarefree monomial
\[
u^\wp=\prod_{i=1}^n\prod_{j=1}^{a_i}x_{ij}.
\]
where the $x_{ij}$ are a new set of variables.

Let $I\subset S$ be a monomial ideal with $\mathcal{G}(I)=\{u_1,\ldots,u_m\}$. The {\em polarization} of $I$ to be the ideal $I^\wp=(u_1^\wp, \ldots,u_m^\wp)$.

\begin{Proposition}
\label{pol}
Let $I$ be a monomial ideal. Then $\HS_j(I^\wp)=\HS_j(I)^\wp$ for all $j$.
\end{Proposition}

\begin{proof}
Let $\FF$ be the minimal multigraded free $S$-resolution of $I$. By using the notation as in the proof of Proposition~\ref{localization}, we may write $F_i=\Dirsum_{j}(\xb^{\ab_{ij}})$ for $i=0,\ldots, \projdim(I)$. It follows from a result of Sbarra \cite[Corollary 4.8]{Sb} that $I^\wp$ admits  the minimal multigraded free resolution $\GG$ (over the polynomial ring with the required variables), where $\G_i= \Dirsum_{j}((\xb^{\ab_{ij}})^\wp)$ for $i= 0.\ldots, \projdim (I^\wp)$. This yields the desired conclusion.
\end{proof}

As an example we apply Proposition \ref{pol} to compute a certain homological shift ideal of a whisker graph. Let G be a finite simple graph on the vertex set $[n]=\{1, \ldots , n\}$. The {\em whisker graph}
$G^*$ of $G$  is the graph with the vertex set $V(G^*)=\{1, \ldots, n\} \cup \{1', \ldots, n'\}$ and the edge set $E(G^*)=E(G) \cup \{ \{1, 1'\}, \{2, 2'\}, \ldots, \{n, n'\} \}$.  We identify the edge ideal $I(G^*)$ of $G^*$ with the monomial ideal $(I(G), x_1y_1,\ldots,x_ny_n)$ in $K[x_1,\ldots,x_n,y_1,\ldots,y_n]$, where $K$ is a field.  Here $I(G)=(x_ix_j\:\; \{i,j\}\in E(G)\}$ is the edge ideal of $G$. We set $I = ( I(G) , x_1^2 , \ldots, x_n^2 )$. Then,  $I(G^*)= I^{\wp}$, where for simplicity  we set $x_i = x_{i1} , y_i = x_{i2}$, for $i= 1, \ldots, n$.

Let $\Delta(G)$ be the independence complex of $G$. Its faces consist of all independent sets $F$, that is, sets  $F\subseteq [n]$ which contain no edge of $G$. It is obvious that the monomials $x_F$ with $F\in \Delta(G)$ form a $K$-basis of $K[x_1,\ldots,x_n]/I$. Let $\Fc(\Delta(G))$ denote the set of facets of $\Delta(G)$. Then it follows that  $(I:\mm)/I$ is generated by the monomials $x_F$ with $F\in \Fc(\Delta(G))$. Thus, Proposition~\ref{last} implies that $\HS_{n-1}(I)=(x_{[n]}x_F\:\; F\in \Fc(\Delta(G)))$. Now we apply Proposition~\ref{pol} and obtain that $\HS_{n-1}(I(G^*))$ is generated by the monomials $(x_{[n]}x_F)^\wp$ with  $F\in \Fc(\Delta(G))$. Note that $(x_{[n]}x_F)^\wp=x_{[n]}y_F$, where $y_F=\prod\limits_{i\in F}y_i$. Hence we have shown

\begin{Corollary}
\label{whisker}
$\HS_{n-1}(I(G^*))=(x_{[n]}y_F\:\; F\in \Fc(\Delta(G)))$.
\end{Corollary}

\section{Homological shifts of $\cb$-bounded strongly stable principal ideals}

Throughout this section, $\cb\in \mathbb{Z}^n$ is an integer vector with non-negative entries.

\begin{Definition}{\em
 A monomial ideal $I\subset S$ is called {\em $\cb$-bounded strongly stable}, if each monomial $u\in \mathcal{G}(I)$ is $\cb$-bounded, and for all $i<j$  for which $x_j|u$  and $x_i(u/x_j)$ is $\cb$-bounded, it follows that $x_i(u/x_j)\in I$.}
 \end{Definition}

Let $u_1,\ldots,u_m\in S$ be $\cb$-bounded monomials. The smallest $\cb$-bounded strongly stable ideal containing $u_1,\ldots,u_m$ is denoted by $B^\cb(u_1,\ldots,u_m)$. A monomial ideal $I$ is called a $\cb$-bounded  strongly  stable principal ideal, if there exists a $\cb$-bounded monomial $u$ such that $I=B^{\cb}(u)$.  The smallest strongly stable ideal containing $u_1,\ldots,u_m$ (with no restrictions on the exponents) is denoted by $B(u_1,\ldots,u_m)$.

\medskip
A well-known example of ideals of this type are the so-called  {\em Veronese type ideals}. Its definition is given at the beginning of section 3. In our terminology these are the ideals of the form $I=B(x_n^d)^{\leq\cb}$.

The main result of this section is the following

\begin{Theorem}
\label{mainHSj}
Let $I$ be  a $\cb$-bounded strongly stable principal ideal. Then $\HS_j(I)$ has linear quotients for all $j$.
\end{Theorem}

The proof of this theorem requires some preparation.

\medskip

Let $u$ and $v$ be $\cb$-bounded monomials of same degree $d$.   Then we write $v\prec_\cb u$ if and only if  $v\in B^\cb(u)$. This is a partial order on the $\cb$-bounded monomials of degree $d$. We also write $v\prec u$ if and only if $v\in B(u)$.

\begin{Lemma}
\label{principalbound}
Let $u\in S$ be a $\cb$-bounded monomial. Then $B^\cb(u)=B(u)^{\leq\cb}$.
\end{Lemma}

\begin{proof}  The inclusion $B^\cb(u)\subseteq B(u)^{\leq\cb}$ is obvious.  Conversely, assume that   $u=x_{i_1}\cdots x_{i_d}$ with $i_1\leq i_2\leq \cdots \leq i_d$, and let  $v\in B(u)^{\leq\cb}$, $v=x_{j_1}\cdots x_{j_d}$  with $j_1\leq j_2\leq \cdots \leq j_d$. We may assume that $v\neq u$.
Let $\ell+1$ be the smallest integer such that $i_{\ell+1}>j_{\ell+1}$, and $u_1=x_{i_{\ell+1}-1}(u/x_{i_{\ell+1}})$.
Then  $u_1$ is $\cb$-bounded.  Otherwise, $i_{\ell+1}-1=i_{\ell}=i_{\ell-1}=\cdots=i_{\ell-c_{\ell}+1}$.
 By the choice of $\ell$ it follows that $j_\ell=j_{\ell-1}= \cdots =j_{{\ell}-c_{\ell}+1}$.
Since $v$ is $\cb$-bounded, it follows that $j_{\ell+1}>j_{\ell}$. Moreover,
$j_\ell=i_\ell$. Then $i_{\ell}<j_{\ell+1}<i_{\ell+1}=i_{\ell}+1$, a contradiction.
This shows that indeed
$u_1$ is $\cb$-bounded, and hence $u_1\prec_\cb u$. Also we have  $v\in B(u_1)^{\leq\cb}$.
In fact, by the same argument,  if $u_1=x_{k_1}\cdots x_{k_d}$, we set $u_2=x_{k_{\ell'+1}-1}(u_1/x_{k_{\ell'+1}})$, where $\ell'+1$ is the smallest integer such that $k_{\ell'+1}>j_{\ell'+1}$, then
$u_2\prec_\cb u_1$ and $v\in B(u_2)^{\leq\cb}$. So we have a chain $u\succ_\cb u_1\succ _\cb u_2\succ_\cb \cdots \succ_\cb u_m$, where $u_m=v$. Thus $v\prec_\cb u$ and then $v\in B^\cb(u)$.

By induction we may assume that $B(u_1)^{\leq\cb}=B^\cb(u_1)$. Therefore,  $v\in B^\cb(u_1)$. The desired conclusion follows since $B^\cb(u_1)\subset B^\cb(u)$.

\end{proof}

An immediate  consequence of Lemma~\ref{restriction} is

\begin{Corollary}
\label{linear}
Let $I$ be a monomial ideal generated in degree $d$. If $I$ has linear resolution, then $I^{\leq\cb}$ has linear resolution.
\end{Corollary}

The following lemma is needed in the sequel.

\begin{Lemma}
\label{criterion}
Let $I$ be a monomial ideal generated in degree $d$ with $\mathcal{G}(I)=\{u_1,\ldots, u_m\}$. Then the following conditions are equivalent:
\begin{enumerate}
\item[(a)] $I$  has linear quotients with respect to $u_1,\ldots, u_m$.
\item[(b)]  the ideals $I_j=(u_1,\ldots,u_j)$ have linear resolution for $j=1,\ldots,m$.
\end{enumerate}
\end{Lemma}

\begin{proof}
For any $j<m$, we obtain the following short  exact sequence
\[
0\to I_j\to I_{j+1}\to (S/(I_j:u_{j+1}))(-d)\to 0,
\]
because $I_{j+1}/I_j\iso (S/(I_j:u_{j+1}))(-d)$. This short exact sequence induces the long exact sequence
\begin{eqnarray}
\label{long}
&&\Tor_i(K, I_{j})_{i+\ell} \to \Tor_i(K, I_{j+1})_{i+\ell}\to \Tor_i(K,S/(I_j:u_{j+1}))_{i+\ell-d}\\
&&\to \Tor_{i-1}(K, I_{j})_{i-1+(\ell+1)}\to \cdots\nonumber
\end{eqnarray}

 (a)\implies(b): We apply induction on $j$. The case $j=1$ is clear. Suppose that the statement holds for $j$. Since
 $I_{j+1}/I_j  \iso u_{j+1}(S/(I_j:u_{j+1}))$, we know that $(S/(I_j:u_{j+1}))(-d)$ has $d$-linear resolution. It follows that $\Tor_i(K,S/(I_j:u_{j+1}))_{i+\ell-d}=0$ for any $\ell\neq d$. Therefore, we obtain that $\Tor_i(K, I_{j+1})_{i+\ell}=0$  for $\ell\neq d$ by long exact sequence (\ref{long}) and  inductive assumption. Thus the assertion follows.

 (b)\implies (a): Since we assume that $I_j$ and  $I_{j+1}$ have $d$-linear resolution it follows from (\ref{long}) that $\Tor_i(K,S/(I_j:u_{j+1}))_{i+\ell-d}=0$ for $\ell\neq d, d-1$. Therefore, $\Tor_i(K,S/(I_j:u_{j+1}))_{i+\ell}=0$ if $\ell\neq 0, - 1$. Since $S/(I_j:u_{j+1})$ is generated in degree $0$, it follows that $\Tor_i(K,(S/(I_j:u_{j+1}))_{i+\ell}=0$  if $\ell=-1$. This shows that $\Tor_i(K,(S/(I_j:u_{j+1}))_{i+\ell}=0$ for $\ell\neq 0$. Hence $I_j:u_{j+1}$ has $1$-linear resolution. Since $I_j:u_{j+1}$ is monomial ideal, this implies that $I_j:u_{j+1}$ is generated by variables.
\end{proof}

By using Lemma~\ref{criterion} and Corollary~\ref{linear} we get

\begin{Proposition}
\label{linearquotients}
Let $I$ be a monomial ideal generated in degree $d$. If $I$ has linear quotients, then $I^{\leq \cb}$ has linear quotients.
\end{Proposition}
\begin{proof} Let $\mathcal{G}(I)=\{u_1,\ldots, u_m\}$ and  assume $I$  has linear quotients with respect to $u_1,\ldots, u_m$.
Then the ideals $I_j=(u_1,\ldots,u_j)$ have linear resolution for $j=1,\ldots,m$,  by Lemma \ref{criterion}.
Let $I^{\leq\cb}=$ where $u_{i_\ell}$'s are the $\cb$-bounded elements of $\mathcal{G}(I)$ and $i_1<i_2<\cdots<i_s$.  By Lemma \ref{criterion}, it is enough to show that $(I^{\leq\cb})_t=(u_{i_1},\ldots,u_{i_t})$ has linear resolution for $t=1,\ldots,s$. We have
$(I^{\leq\cb})_t=(u_{i_1},\ldots,u_{i_t})=J^{\leq\cb}$, where $J=(u_k:\ 1\leq k\leq i_t)$.
Since $I$ has linear quotients and $J=I_{i_t}$, by Lemma \ref{criterion}, $J$ has a linear resolution. So by Corollary \ref{linear}, $J^{\leq\cb}$ has a linear resolution.

\end{proof}
Now we are ready to prove the main result of this section.

\begin{proof}[Proof of Theorem~\ref{mainHSj}]
By Lemma~\ref{principalbound} we have $B^\cb(u)=B(u)^{\leq\cb}$. Therefore, Corollary~\ref{hsjrestriction} implies that  $\HS_j(B^\cb(u))=\HS_j(B(u))^{\leq \cb}$. By a result of Bayati et al \cite{BJT} $\HS_j(B(u))$ has linear quotients. The desired result follows now from  Proposition~\ref{linearquotients}.
\end{proof}

\section{Homological shifts of Veronese type ideals}

Given positive integers  $n$ and $d$ and   an  integer vector $\cb=(c_1,\ldots,c_n)$ with $c_i\geq  0$. The
monomial ideal  $I_{\cb,n,d}\subset S=K[x_1,\ldots,x_n]$ with
\[
\mathcal{G}(I_{\cb,n,d})=\{x_1^{a_1}x_2^{a_2}\cdots x_n^{a_n}\; \mid \; \sum_{i=1}^na_i=d \text{ and  $a_i\leq c_i$ for $i=1,\ldots,n$}\}
\]
is  called a {\em Veronese type ideal}. Observe that $I_{\cb,n,d}=0$ if $\sum_{i=1}^nc_i<d$. We therefore assume for the rest of this section that $\sum_{i=1}^nc_i\geq d$, unless otherwise stated.

\begin{Proposition}
\label{veroneseborel}
Let $\cb$ be an integer vector with non-negative components and $d$ and $n$ be positive integers with $I_{\cb,n,d}\neq 0$.  Then $I_{\cb,n,d}=B^\cb(u)$ for some monomial $u$.
\end{Proposition}

\begin{proof}
By Lemma \ref{principalbound}, it is enough to find a monomial $u$ such that $I_{\cb,n,d}=B(u)^{\leq\cb}$. Let $\cb=(c_1,\ldots,c_n)$. First assume that $c_n>d$. We show that $I_{\cb,n,d}=B(u)^{\leq\cb}$, where $u=x_n^d$. Clearly  $B(u)^{\leq\cb}\subseteq I_{\cb,n,d}$. Now, let $v=x_1^{a_1}\cdots x_n^{a_n}\in I_{\cb,n,d}$. If $v=u$, then $v\in B(u)^{\leq\cb}$. Let $v\neq u$. Since $\sum_{i=1}^ n a_i=d$, we have $a_n\leq d$ and then $v\in B(u)^{\leq\cb}$. So $I_{\cb,n,d}=B(u)^{\leq\cb}$.

Now, let $c_n\leq d$. We set $m=\max\{j\:\; \sum\limits_{i=n-j+1}^nc_i\leq d\}$ and $r=d-\sum\limits_{i=n-m+1}^nc_i$. Note that since $c_n\leq d$, $m$ is well-defined and $m\geq 1$. If $m=n$, then $\sum\limits_{i=1}^nc_i=d$, because $I_{\cb,n,d}\neq 0$. Thus $I_{\cb,n,d}=(u)=B(u)^{\leq\cb}$, where $u=x_1^{c_1}\cdots x_n^{c_n}$ and we are done. So we may assume that $m<n$. Set $u=x_{n-m}^r\prod\limits_{i=n-m+1}^nx_i^{c_i}$, where $r=d-\sum\limits_{i=n-m+1}^n c_i$.
Clearly  $B(u)^{\leq\cb}\subseteq I_{\cb,n,d}$. Now, let $v=x_1^{a_1}\cdots x_n^{a_n}\in I_{\cb,n,d}$. Then
\[
v=\begin{cases}
      x_{n-m}^{a_{n-m}-r} \prod\limits_{i=1}^ {n-m-1} x_i^{a_i} (u/\prod\limits_{i=n-m+1}^ n x_i^{c_i-a_i}), & \text{if $r<a_{n-m}$}, \\
     \prod\limits_{i=1}^ {n-m-1} x_i^{a_i} (u/(x_{n-m}^{r-a_{n-m}}.\prod\limits_{i=n-m+1}^ n x_i^{c_i-a_i})), & \text{if $r\geq a_{n-m}$}.
   \end{cases}
 \]
Note that $(a_{n-m}-r)+\sum\limits_{i=1}^{n-m-1} a_i=d-\sum\limits_{i=n-m+1}^{n} a_i-r= \sum\limits_{i=n-m+1}^n(c_i-a_i)$, and $(r-a_{n-m})+\sum\limits_{i=n-m+1}^ n  (c_i-a_i)=d-\sum\limits_{i=n-m}^{n} a_i=\sum\limits_{i=1}^{n-m-1} a_i$. So $v\in B(u)$. Since $v$ is
 $\cb$-bounded, $v\in B(u)^{\leq\cb}$.
\end{proof}

Proposition~\ref{veroneseborel} together with Theorem~\ref{mainHSj} imply

\begin{Corollary}\label{}
The homological shift ideals $\HS_j(I_{\cb,n,d})$ have linear quotients.
\end{Corollary}

\medskip

In what follows  we want to determine the ideals $\HS_j(I_{\cb,n,d})$ explicitly.For this purpose we introduce the following notation. Let $I\subset S$ be an equigenerated monomial ideal,  and let $\ell<n$ be a positive integer. Then we let $I_{>\ell}$ be the ideal  with
\[
\mathcal{G}(I_{>\ell})=\{u\in \mathcal{G}(I)\:\; |\supp(u)|>\ell\}.
\]

Recall that a monomial ideal $I$ is called  polymatroidal if for any two monomials $u=x_1^{a_1}\cdots x_n^{a_n}$ and $u=x_1^{b_1}\cdots x_n^{b_n}$ belonging to $\mathcal{G}(I)$, and for each $i$ with $a_i>b_i$, one has $j$ with $a_j<b_j$ such that $x_ju/x_i\in \mathcal{G}(I)$.
\begin{Theorem}
\label{supp}
For all $\ell\geq 0$ we have:
\begin{enumerate}
\item[(a)] $\HS_\ell(I_{\cb,n,d})=(I_{\cb,n,d+\ell})_{>\ell}$.

\item[(b)] $\HS_\ell(I_{\cb,n,d})$ is a polymatroidal ideal.
\end{enumerate}
\end{Theorem}

\begin{proof} (a) Case $\ell=0$ is trivial. Assume that $\ell\geq 1$. By Proposition \ref{veroneseborel}, Corollary~\ref{hsjrestriction} and Example~\ref{exm1}(c), we obtain that
\[\HS_\ell(I_{\cb,n,d})=(x_Fu\:\;  u\in \mathcal{G}(I_{\cb,n,d}),\; \max(F)< m(u),|F|=\ell \text{ and $x_Fu$ is $\cb$-bounded}),
 \]
 where $F\subset [n]$, $\max(F) =\max\{i\:\; i\in F\}$ and $m(u)=\max\{i\:\; \text{ $x_i$ divides $u$}\}$.

Let $x_Fu\in \HS_\ell(I_{\cb,n,d})$, then $|\supp(x_Fu)|\geq |F|+1=\ell+1$ because of  $\max(F)< m(u)$.
The degree of $x_Fu$ equals to $d+|F|=d+\ell$, it follows that $x_Fu\in (I_{\cb,n,d+\ell})_{>\ell}$
since $x_Fu$ is $\cb$-bounded.

Conversely, let  $v\in \mathcal{G}((I_{\cb,n,d+\ell})_{>\ell})$, then $v$ is $\cb$-bounded. We may write $v$ as $v=x_{j_1}^{\alpha_{j_1}}\cdots x_{j_t}^{\alpha_{j_t}}$ with
 $1\leq j_1<j_2<\cdots<j_{t}\leq n$ and $t>\ell$ and $\alpha_{j_k}>0$ for $k=1,\ldots,t$.
Set $F=\{j_1,j_2,\ldots,j_\ell\}$ and
$u=x_{j_1}^{\alpha_{j_1-1}}\cdots x_{j_{\ell}}^{\alpha_{j_{\ell}-1}}x_{j_{\ell+1}}^{\alpha_{j_{\ell+1}}}\cdots x_{j_t}^{\alpha_{j_t}}$, then
 $|F|=\ell$, $\max(F)< m(u)$ and  $u$ is $\cb$-bounded of  degree  $d$. This implies that $u\in \mathcal{G}(I_{\cb,n,d})$. Hence  $v=x_Fu\in \HS_\ell(I_{\cb,n,d})$.

(b) The ideal $I_{\cb,n,d}$ is a polymatroidal ideal satisfying the strong exchange property, see \cite[Example 2.6]{HH1}. Therefore the desired result follows from part (a) and the next proposition.
\end{proof}

A polymatroidal ideal $I$ is said to satisfy the {\em strong exchange property} if for all $u=x_1^{a_1}\cdots x_n^{a_n}$ and $v=x_1^{b_1}\cdots x_n^{b_n}$ in $\mathcal{G}(I)$ and all $i$ and $j$ with $a_i >b_i$ and $a_j<b_j$ it follows that $x_j(u/x_i)\in \mathcal{G}(I)$.
We may apply the following result to complete the proof of Theorem \ref{supp}(b).

\begin{Proposition}
\label{strongpoly}
Let $I$ be a polymatroidal ideal  satisfying the strong exchange property. Then $I_{>\ell}$ is polymatroidal.
\end{Proposition}

\begin{proof}
Let $u,v\in \mathcal{G}(I_{>\ell})$, $u=x_1^{a_1}\cdots x_n^{a_n}$ and $v=x_1^{b_1}\cdots x_n^{b_n}$. Assume that $a_i>b_i$. We have to find $j$ such that $b_j>a_j$ and $x_j(u/x_i)\in I$ and $|\supp(x_j(u/x_i))|>l$. We may assume that $i=1$. Since $I$ is polymatroidal, there exists $j$ such that $b_j>a_j$ and $x_j(u/x_1)\in I$.  If $a_1>1$, then $|\supp(x_j(u/x_1))|\geq |\supp(u/x_1)|=|\supp(u)|>\ell$  and so $x_j(u/x_1)\in I_{>\ell}$. If $a_1=1$ and  $|\supp(u)|>{\ell+1}$, then  $|\supp(x_j(u/x_1))\geq |\supp(u)|-1>\ell$ and we are done. Finally if $a_1=1$ and $|\supp(u)|=\ell+1$, we may assume that $u=x_1x_2^{a_2}\cdots x_{\ell}^{a_\ell}$.  Then $v=x_2^{b_2}\cdots x_n^{b_n}$. Since $|\supp(v)|>\ell$, there exists $j>\ell$ with $b_j\neq 0$.  Since $I$ satisfies the strong exchange property, it follows that $x_j(u/x_1)\in I$. Moreover $|\supp(x_j(u/x_1))|=|\supp(u)|>\ell$. This completes the proof.
\end{proof}

\begin{Remark}
{\em If the polymatroidal ideal $I$ does not satisfy the strong exchange property, then $I_{>\ell}$ may not be polymatroidal. Indeed, the ideal $$I=(x_1x_2x_3x_4, x_2^2x_4^2, x_2x_3x_4^2, x_1x_2^2x_4)$$ is polymatroidal, but does not satisfy the strong exchange property, while the ideal
$I_{>2}=( x_1x_2x_3x_4,  x_2x_3x_4^2, x_1x_2^2x_4)$ is not polymatroidal. Note however that $I_{>2}$ nevertheless has linear resolution.
}
\end{Remark}

It is well-known that a polymatroidal ideal $I$ satisfies  the strong exchange property if and only if $I$ is essentially of Veronese type, which means that  $I$ is of Veronese type, up to multiplication by a monomial. Also one can easily see that for any monomial $m$, $\HS_j(mI)=m\HS_j(I)$ for all $j$.  Thus Theorem \ref{supp}(b) implies

\begin{Corollary}
Let $I$ be a  polymatroidal ideal satisfying  the strong exchange property. Then $\HS_j(I)$ is a  polymatroidal ideal for all $j$.
\end{Corollary}

\section{Homological shifts of edge ideals}

In this section we study the homological shift ideals  of edge ideals which have linear resolution. Let $G$ be a graph with the vertex set $V(G)$ and the edge set $E(G)$. Recall that a perfect elimination ordering for $G$, is a total ordering on $x_1>x_2>\cdots>x_n$ on $V(G)$
such that $N_{G_i}(x_i)$ induces a complete subgraph of $G$, where $G_i$ is the induced subgraph of $G$ on the vertex set $\{x_i,x_{i+1},\ldots,x_n\}$ and $N_G(x)$ denotes the set of vertices  which are adjacent to $x$ in $G$. Also we set $N_G[x]=N_G(x)\cup \{x\}$.

\begin{Theorem}\label{chordal}
 Let $I$ be an edge ideal with linear resolution and $G$ be the chordal graph for which $I=I(G^c)$.
Let $x_1>x_2>\cdots>x_n$ be a perfect elimination ordering for $G$. Then
\[
\HS_k(I)=(x_{i_1}x_{i_2}\cdots x_{i_{k+2}}\:\; 1\leq i_1<i_2<\cdots<i_{k+2}\leq n,  \text{ and there exists}
\]
\[ \hskip 40mm \text{ $t<k+2$ such that $\{x_{i_t},x_{i_{\ell}}\}\notin E(G)$ for all $t<\ell\leq k+2$}).
\]
\end{Theorem}

\begin{proof}
Let $x_1>x_2>\cdots>x_n$ be a perfect elimination ordering for $G$. Then $I$  has linear quotients with the lex order on $\mathcal{G}(I)$ induced by $x_1>x_2>\cdots>x_n$. Indeed one can see that
$I=J+x_1 (x_{l_1},\ldots,x_{l_s})$, where $J=I((G\setminus x_1)^c)$ and $\{x_{l_1},\ldots,x_{l_s}\}=V(G)\setminus N_G[x_1]$ with $l_1<\cdots < l_s$. By induction hypothesis let $f_1>\cdots>f_r$ be the order of linear quotients on the minimal generators of $J$. Then $x_1x_{l_1}>\cdots> x_1x_{l_s}>f_1>\cdots>f_r$ is an order of linear quotients of $I$. Note that  $N_G[x_1]$ induces a complete subgraph of $G$. Thus for any $1\leq j\leq r$, $\supp(f_j)\cap \{x_{l_1},\ldots,x_{l_s}\}\neq \emptyset$, since $\supp(f_j)$ is an independent set of $G\setminus x_1$. So for any $f_j$, there exists $x_{l_p}$ which divides $f_j$ and then $x_1x_{l_p}:f_j=x_1$ and $x_1|(x_1x_{l_i}:f_j)$ for any $1\leq i\leq s$. This means that $\set_I(f_j)=\set_J(f_j)\cup\{1\}$. Also it is clear that $\set_I(x_1x_{l_i})=\{l_1,\ldots,l_{i-1}\}$ for any $i$.
By induction on $n$, we show that for any $x_ix_j\in I$, where $i<j$, $\set(x_ix_j)=\{1,2,\ldots,i-1,j_1,\ldots,j_s\}$, where
$\{j_1,\ldots,j_s\}=\{k: x_k\notin N_G[x_i], i<k<j\}$. For $i=1$, the equality is already shown. Now, assume that $i\geq 2$. Then $x_ix_j\in J$.
Since $x_2>\cdots>x_n$ is a perfect elimination ordering for $G\setminus x_1$, by induction hypothesis we may assume that $\set_J(x_ix_j)=\{2,3,\ldots,i-1,j_1,\ldots,j_s\}$, where
$\{j_1,\ldots,j_s\}=\{k: x_k\notin N_{G\setminus x_1}[x_i], i<k<j\}$. Note that for any $k>i$, $x_k\notin N_{G\setminus x_1}[x_i]$ is equivalent to $x_k\notin N_{G}[x_i]$.
Thus $\set_I(x_ix_j)=\set_J(x_ix_j)\cup\{1\}=\{1,2,\ldots,i-1,j_1,\ldots,j_s\}$, where $\{j_1,\ldots,j_s\}=\{k: x_k\notin N_G[x_i], i<k<j\}$, as desired.

By (\ref{eq1}), any minimal generator of $\HS_k(I)$ is of the form $u=x_ix_jx_F$, where $x_ix_j\in I$ with $i<j$ and $F\subseteq \set_I(x_ix_j)= \{1,2,\ldots,i-1,j_1,\ldots,j_s\}$ such that $|F|=k$. We may write $u=x_{i_1}x_{i_2}\cdots x_{i_{k+2}}$, where $1\leq i_1<i_2<\cdots<i_{k+2}\leq n$.
Since  $i<j$, we have $i=i_t$ for some $t<k+2$. For any $\ell$ with $t<\ell\leq k+2$, we have  $i_{\ell}>i_t=i$. Hence $x_{i_{\ell}}|x_jx_F$, and then  $i_{\ell}\in \{j_1,\ldots,j_s\}\cup\{j\}$. Therefore $x_{i_{\ell}}\notin N_G[x_{i_t}]$, i.e., $\{x_{i_t},x_{i_{\ell}}\}\notin E(G)$. Conversely, consider an element $v=x_{i_1}x_{i_2}\cdots x_{i_{k+2}}$, where $1\leq i_1<i_2<\cdots<i_{k+2}\leq n$ such that there exists $t<k+2$ with $\{x_{i_t},x_{i_{\ell}}\}\notin E(G)$ for all $t<\ell\leq k+2$. Then $v=x_{i_t}x_{i_{t+1}} x_F$, where $F= \{i_1,\ldots,i_{t-1},i_{t+2},\ldots,i_{k+2}\}$. Note that by our assumption, $x_{i_{t+1}},x_{i_{t+2}},\ldots,x_{i_{k+2}}\notin N_G[x_{i_t}]$. Thus $x_{i_t}x_{i_{t+1}}\in I$ and $F\subseteq \set(x_{i_t}x_{i_{t+1}})$ with $|F|=k$. So $v\in \HS_k(I)$.   The proof is complete.
\end{proof}


\begin{Proposition}\label{pathlinear}
Let $n$ be any positive integer and $I=I(P_n^c)$. Then $\HS_k(I)$ has linear quotients for any $k\geq 1$.
\end{Proposition}

\begin{proof}
Let $P_n:x_1,x_2,\ldots,x_n$ be a path graph and $I=I(P_n^c)$. We have $x_ix_j\in I$, if and only if $j-i\geq 2$. So by Theorem \ref{chordal},
\[
\HS_k(I)=(x_{i_1}x_{i_2}\cdots x_{i_{k+2}}\:\; 1\leq i_1<i_2<\cdots<i_{k+2}\leq n,  \text{ and there exists}
\]
\[ \hskip 90mm \text{ $t<k+2$ such that $i_{t+1}-i_t\geq 2$}).
\]
We show that $\HS_k(I)$ has linear quotients with the lex order induced by $x_1>\cdots>x_n$.
Let $u=x_{i_1}x_{i_2}\cdots x_{i_{k+2}}$ and $v=x_{j_1}x_{j_2}\cdots x_{j_{k+2}}$ be two minimal generators of $\HS_k(I)$ such that
$1\leq i_1<i_2<\cdots<i_{k+2}\leq n$ and $1\leq j_1<j_2<\cdots<j_{k+2}\leq n$ and $u>_{\lex} v$. So there exists $1 \leq d\leq k+2$ such that
$i_r=j_r$ for $r<d$ and $i_d<j_d$. If $d=k+2$, then $u:v=x_{i_d}$ and we are done. So let $d<k+2$. Set $w=(v/x_{j_d}) x_{i_d}$. Then $w=x_{i_1}x_{i_2}\cdots x_{i_d}x_{j_{d+1}}\cdots x_{j_{k+2}}$.
 Since $i_d<j_d<j_{d+1}$, we have $j_{d+1}-i_d\geq 2$. So $w$ is a minimal generator of $\HS_k(I)$. Clearly $x_{i_d}|(u:v)$,  $w>_{\lex} v$ and $w:v=x_{i_d}$.
\end{proof}

A graph $G$ on the vertex set $\{x_1,\ldots,x_n\}$ is called a {\em proper interval graph}, if for all $i<j$, $\{x_i, x_j\}\in E(G)$ implies that the induced subgraph of $G$ on $\{x_i, x_{i+1}, \dots ,x_j\}$ is a complete graph. Path graph is a simple example of a proper interval graph. Any proper interval graph $G$ can be described as  $G=L_1\cup\cdots \cup L_p$, where each $L_i$ is a maximal complete subgraph of $G$ and the labeling of $V(G)$ is such that for any $i\neq j$, if
$x_i\in L_r$ and $x_j\in L_s$ and $r<s$, then $i<j$.
In the next lemma we determine when the homological shift ideal of $I(G^c)$ is non-zero, when $G$ is a proper interval graph.

\begin{Lemma}\label{interval}
Let $G=L_1\cup\cdots \cup L_p$ be a proper interval graph, $s=\min\{|L_i\cap L_{i+1}|:\  1\leq i\leq p-1\}$ and $I=I(G^c)$. Then $\HS_k(I)\neq (0)$ if and only if $k\leq n-s-2$. In particular, $\projdim(I)=n-s-2$.
\end{Lemma}

\begin{proof}
By Theorem \ref{chordal}, we have
\[
\HS_k(I)=(x_{i_1}x_{i_2}\cdots x_{i_{k+2}}\:\; 1\leq i_1<i_2<\cdots<i_{k+2}\leq n,  \text{ and there exists}
\]
\[ \hskip 40mm \text{ $t<k+2$ such that $\{x_{i_t},x_{i_{\ell}}\}\notin E(G)$ for all $t<\ell\leq k+2$}).
\]
First assume that $k\leq n-s-2$. Let $T=L_d\cap L_{d+1}=\{x_{r+1},\ldots,x_{r+s}\}$ be of size $s$ in $G$. Since $k+s+2\leq n$, we set $u=\prod\limits_{i=1}^{k+s+2} x_i/\prod\limits_{i=r+1}^{r+s} x_i$. Then $\deg(u)=k+s+2-s=k+2$. Also  $\{x_r,x_{r+s+1}\}\notin E(G)$. Otherwise
$\{x_r,\ldots,x_{r+s+1}\}\subseteq L_q$ for some $1\leq q\leq p$. Note that $q\neq d,d+1$, because $x_{r+s+1}\in L_q\setminus L_d$ and $x_r\in L_q\setminus L_{d+1}$. But since $x_r\in L_d$ and $x_{r+s+1}\in L_{d+1}$, we should have $d<q<d+1$, which is a contradiction.
Thus $\{x_{r},x_{r+s+1}\}\notin E(G)$. Since $G$ in an interval graph, $\{x_{r},x_{\ell}\}\notin E(G)$ for any $\ell>r+s+1$. So $u\in \HS_k(I)$ and $\HS_k(I)\neq 0$.

Now, assume that $\HS_k(I)\neq 0$ and
$u=x_{i_1}x_{i_2}\cdots x_{i_{k+2}}\in \HS_k(I)$. Let $t<k+2$ be such that $\{x_{i_t},x_{i_{\ell}}\}\notin E(G)$ for all $t<\ell\leq k+2$.
Since  $\{x_{i_t},x_{i_{t+1}}\}\notin E(G)$, we have $x_{i_t}\in L_m$ and $x_{i_{t+1}}\in L_{m'}$, where $m\neq m'$. So $m<m'$. Let $\lambda$ be the smallest integer with $x_{\lambda}\in L_{m+1}\setminus L_m $. Since $|L_m\cap L_{m+1}|\geq s$, we have $\lambda-i_t>s$.
Since $m'\geq m+1$, one has $i_{t+1}\geq \lambda$. Hence $i_{t+1}-i_t\geq \lambda-i_t>s$. Thus $\supp(u)\subseteq \{1,\dots,n\}\setminus \{i_t+1,i_t+2,\ldots,i_t+s\}$ and so $k+2=|\supp(u)|\leq n-s$.
\end{proof}

\begin{Corollary}\label{pathk}
Let $n$ and $k$ be positive integers and $I=I(P_n^c)$. Then $\HS_k(I)\neq (0)$ if and only if $k\leq n-3$.
\end{Corollary}


\begin{Proposition}
\label{projdim}
Let $n$ and $k$ be positive integers and $I=I(P_n^c)$. If $\HS_k(I)\neq (0)$, then $\projdim(\HS_k(I))=n-k-2$.
\end{Proposition}

\begin{proof}
By \cite[Lemma 1.5]{HT} and Proposition \ref{pathlinear}, one has $$\projdim(\HS_k(I))=\max\{|\set(u)|:\ u\in \mathcal{G}(\HS_k(I))\}.$$
Note that by Corollary \ref{pathk}, $k\leq n-3$.
Also $\HS_k(I)$ has linear quotients with the lex order induced by $x_1>\cdots>x_n$ as was shown in the proof of Proposition \ref{pathlinear}.
Let $u=x_{i_1}x_{i_2}\cdots x_{i_{k+2}}\in \HS_k(I)$ such that $i_1<i_2<\cdots<i_{k+2}$. Let $1\leq t\leq k+1$ be the smallest integer  such that $i_{t+1}-i_t\geq 2$ (see Theorem \ref{chordal}). Set $A=[1,i_{k+2}]\setminus \{i_1,i_2,\ldots,i_{k+2}\}$. We show that
\[
\set(u) = \left\{
\begin{array}{ll}
A,  &   \text{if $t<k+1$,}\\
A\setminus \{i_{k+1}+1\}, & \text{if $t=k+1$.}\\
\end{array}
\right. \]
Clearly $\set(u)\subseteq A$, since all minimal generators of $\HS_k(I)$ are squarefree. Let $t<k+1$. Consider $j\in A$. If $j<i_1$, then for $v= x_j(u/x_{i_1})$, we have $v\in \mathcal{G}(\HS_k(I))$, $v>_{lex} u$ and $v:u=x_j$. Thus $j\in \set(u)$. If $j>i_1$, then there exists $\ell$ such that $i_{\ell}<j<i_{\ell+1}$. Thus for $v=x_j(u/x_{i_{\ell+1}}) $, we have $v>_{lex} u$ and $v:u=x_j$. We show that  $v\in \mathcal{G}(\HS_k(I))$. Indeed if $\ell\leq t$, then $\ell+2\leq t+2\leq k+2$ and so $x_j,x_{i_{\ell+2}}|v$. Since $i_{\ell+2}>i_{\ell+1}>j$, one has $i_{\ell+2}-j\geq 2$ and we are done. If $\ell>t$, then $x_{i_t},x_{i_{t+1}}|v$ and $i_{t+1}-i_t\geq 2$ and so $v\in \mathcal{G}(\HS_k(I))$.
Therefore $j\in set(u)$. Thus $A=\set(u)$.

Now, assume that $t=k+1$. By contradiction assume that $i_{k+1}+1\in \set(u)$. Then there exists $v\in \mathcal{G}(\HS_k(I))$ such that  $v>_{lex} u$ and $v:u=x_{i_{k+1}+1}$. So $v=x_{i_{k+1}+1}(u/x_{i_p}) $, where $i_p>i_{k+1}+1$. Therefore $p=i_{k+2}$ and then $v=x_{i_1}x_{i_2}\cdots x_{i_{k+1}}x_{i_{k+1}+1}$. But then by our assumption on $t$, $i_{s+1}-i_s=1$ for any $s<t=k+1$, so $v\notin \HS_k(I)$, a contradiction. Thus $i_{k+1}+1\notin \set(u)$ and $\set(u)\subseteq A\setminus \{i_{k+1}+1\}$.
Conversely, consider $j\in  A\setminus \{i_{k+1}+1\}$.
If $j<i_1$, then for $v=x_j(u/x_{i_1}) $, we have $v\in \mathcal{G}(\HS_k(I))$, $v>_{lex} u$ and $v:u=x_j$. If $j>i_1$, then there exists $\ell$ such that $i_{\ell}<j<i_{\ell+1}$. Thus for $v=x_j(u/x_{i_{\ell+1}}) $, we have $v>_{lex} u$ and $v:u=x_j$. Note that  $v\in \mathcal{G}(\HS_k(I))$. Indeed if $\ell<t$, then $\ell+2\leq t+1=k+2$. Thus $x_j,x_{i_{\ell+2}}|v$ and  $i_{\ell+2}>i_{\ell+1}>j$ and so $i_{\ell+2}-j\geq 2$. If $\ell>t$, then $x_{i_t},x_{i_{t+1}}|v$ and $i_{t+1}-i_t\geq 2$.
If $\ell=t=k+1$, then $i_{k+1}<j<i_{k+2}$. Since $j\neq i_{k+1}+1$, one has $j-i_{k+1}\geq 2$. Also $x_j,x_{i_{k+1}}|v$. Therefore $v\in \mathcal{G}(\HS_k(I))$ and so $j\in set(u)$.

Thus for a minimal generator $u$ of $\HS_k(I)$, $|\set(u)|$ is maximal when $|\set(u)|=|A|=i_{k+2}-k-2$, where  $i_{k+2}=n$.  This happens when we choose for example $i_1=1$ and $i_{2}=3$ which implies that $t=1<k+1$.
\end{proof}

\begin{Corollary}
Let $n$ and $k$ be  positive integers and $I=I(P_n^c)$ such that $\HS_k(I)\neq (0)$. Then $\projdim(\HS_{k+1}(I))<\projdim(\HS_k(I))$.
\end{Corollary}

For a set $X=\{x_1,\ldots,x_n\}$, we abuse the notation and set $K[X]=K[x_1,\ldots,x_n]$.
A monomial ideal $I$ in $R=K[X]$ is called {\em vertex splittable} if it can be obtained by the following recursive procedure.
\begin{itemize}
\item[(i)] If $u$ is a monomial and $I=(u)$, $I=(0)$ or $I=R$, then $I$ is a vertex splittable ideal.
\item[(ii)] If there is a variable $x\in X$ and vertex splittable ideals $I_1$ and $I_2$ of $K[X\setminus \{x\}]$ so that $I=xI_1+I_2$, $I_2\subseteq I_1$ and $\mathcal{G}(I)$ is the disjoint union of $\mathcal{G}(xI_1)$ and $\mathcal{G}(I_2)$, then $I$ is a vertex splittable ideal.
\end{itemize}
With the above notations if $I=xI_1+I_2$ is a vertex splittable ideal, then $xI_1+I_2$ is called a {\em vertex splitting} for $I$ and $x$ is called a {\em splitting vertex} for $I$.

We expect that all homological shift ideals of edge ideals with linear resolution have linear quotients. In the following theorem this is shown for the first homological shift ideal. Indeed we show that the first homological shift ideal is vertex splittable which by \cite[Theorem 2.4]{MK} this implies that it has linear quotients.

\begin{Theorem}
\label{linear quotients}
Let $G$ be a chordal graph and $I=I(G^c)$. Then $\HS_1(I)$ is a vertex splittable ideal and therefore it has linear quotients.
\end{Theorem}

\begin{proof}
Let $V(G)=\{x,x_1,\ldots,x_n\}$ and $x$ be a simplicial vertex of $G$ with $N_G(x)=\{x_1,\ldots,x_m\}$. We know from \cite[Theorem 2.3]{MK} that $I$ is vertex splittable and $I=J+xK$ is a splitting for $I$, where $J=I((G\setminus x)^c)$ and $K=(x_{m+1},\ldots,x_n)$. Also by \cite[Theorem 2.8]{MK}, if $\mathcal{G}(J)=(g_1,\ldots,g_s)$, then $xx_{m+1},\ldots,xx_n,g_1,\ldots,g_s$ is an order of linear quotients for $I$.
Thus for any $m+1\leq j\leq n$, $\set(xx_j)=\{x_{m+1},x_{m+2},\ldots,x_{j-1}\}$. So by \cite[Lemma 1.5]{HT},
$\HS_1(I)=(ux_F:\ F\subseteq \set(u), |F|=1)=(xx_jx_k:\ m+1\leq j\leq n,\  m+1\leq k\leq j-1)+(g_ix_{\ell}:\ x_{\ell}\in \set_I(g_i))$. One can see that  $\set_I(g_i)=\set_J(g_i)\cup\{x\}$ for any $i$. Therefore $\HS_1(I)=x(x_kx_j:\ m+1\leq k<j\leq n)+x(g_1,\ldots,g_s)+(g_ix_{\ell}:\ x_{\ell}\in \set_J(g_i))=xL+\HS_1(J)$, where $L=(x_kx_j:\ m+1\leq k<j\leq n)+(g_1,\ldots,g_s)$. We show that $\HS_1(I)=xL+\HS_1(J)$ is indeed an splitting for $\HS_1(I)$.
Note that $\HS_1(J)\subseteq J\subseteq L$. By induction on the number of vertices of the graph, we may assume that $\HS_1(J)$ is vertex splittable. It remains to show that $L$ is vertex splittable. Let $H$ be a graph with the vertex set $V(G\setminus x)$ and the edge set $\{\{x_i,x_j\}\in E(G\setminus x):\ \min\{i,j\}\leq m\}$. We show that $L=I(H^c)$.
We have $I(H^c)=(x_ix_j:\ i\neq j, \{x_i,x_j\}\notin E(G\setminus x))+(x_ix_j:\  m+1\leq i<j\leq n)$. Note that $(x_ix_j:\  i\neq j, \{x_i,x_j\}\notin E(G\setminus x))=(g_1,\ldots,g_s)$. Thus $I(H^c)=L$. Since $H$ is a chordal graph with less than $|V(G)|$ vertices, by induction hypothesis $I(H^c)$ is vertex splittable.
\end{proof}

\medskip

\hspace{-6mm} {\bf Acknowledgments}

 \vspace{3mm}
\hspace{-6mm}  This research is supported by the National Natural Science Foundation of China (No.11271275) and  by foundation of the Priority Academic Program Development of Jiangsu Higher Education Institutions.

\end{document}